\documentclass[11pt]{amsart}

\usepackage{xcolor}
\usepackage[T1]{fontenc}
\usepackage{amsmath,amssymb,mathtools}
\usepackage{hyperref}
\usepackage[capitalize]{cleveref}

\newtheorem{theorem}{Theorem}[section]
\newtheorem{proposition}[theorem]{Proposition}
\newtheorem{lemma}[theorem]{Lemma}
\newtheorem{corollary}[theorem]{Corollary}
\theoremstyle{definition}

\newtheorem{remark}[theorem]{Remark}

\newcommand{\R}{\mathbb R}
\newcommand{\one}{\mathbf 1}
\newcommand{\ip}[2]{\left\langle #1,#2\right\rangle}

\title[{Weinstock Inequality on Regular Trees}]{Weinstock Inequality on Regular Trees}

\author{Lili Wang}
\address{Lili Wang: School of Mathematics and Statistics, Key Laboratory of Analytical Mathematics and Applications (Ministry of Education), Fujian Key Laboratory of Analytical Mathematics and Applications (FJKLAMA), Fujian Normal University, Fuzhou 350117, China}
\email{liliwang@fjnu.edu.cn}

\author{Tao Wang}
\address{Tao Wang: Beijing International Center for Mathematical Research, Peking University, Beijing 100871, China}
\email{taowang25@pku.edu.cn}

\date{\today}
\subjclass[2020]{05C05, 05C50, 47A75}
\keywords{Steklov eigenvalue, regular tree, Dirichlet-to-Neumann operator, Weinstock inequality}

\begin{document}

\begin{abstract}
Let $T_n$ be the infinite $n$-regular tree, $n\ge3$. We prove that every finite connected vertex set $\Omega\subset T_n$ satisfies the sharp inequality
\[
\sigma_1(\Omega)\le \frac{n}{(n-1)|\Omega|+1}.
\]
Equality holds if and only if $\Omega$ is a ball. Since
\[
|\delta\Omega|=(n-2)|\Omega|+2,
\]
the result is equivalently a sharp upper bound at fixed external boundary cardinality, and hence a discrete Weinstock inequality on $T_n$.
\end{abstract}
 
\maketitle

\section{Introduction}
For a finite graph with boundary, harmonic extension determines a finite-dimensional Dirichlet-to-Neumann operator: boundary values are sent to the outward normal derivatives of their harmonic extensions. This operator is the discrete analogue of the classical Dirichlet-to-Neumann map on manifolds with boundary; see \cite{GP2017,CGGS2024}. The discrete Steklov problem was introduced and studied in \cite{HHW2017,HM2020}. 

The starting point of the present work is Weinstock's theorem \cite{Weinstock1954}: among simply connected planar domains of prescribed perimeter, the disk uniquely maximizes the first nonzero Steklov eigenvalue. Brock's reciprocal-sum inequality gives the complementary fixed-volume statement for bounded Lipschitz domains in Euclidean space, with the ball as extremizer \cite{Brock2001}. Let $T_n$ be the infinite $n$-regular tree, $n\ge3$. For a finite connected set $\Omega\subset T_n$, let $\delta\Omega$ denote its vertex boundary. Throughout, for a finite set $X$, we write $|X|$ for its cardinality. Boundary and volume are tied by the exact identity (see Lemma~\ref{lem:boundary-volume}) 
\[
|\delta\Omega|=(n-2)|\Omega|+2.
\]
Thus fixing the boundary cardinality is equivalent to fixing $|\Omega|$. For $r\ge0$, we denote by $B_r$ the ball of radius $r$ in $T_n$; its centre is immaterial because $T_n$ is homogeneous.

Related estimates are known in other ambient settings. Han and Hua obtained an upper bound for the first nonzero Steklov eigenvalue of finite domains in $\mathbb{Z}^d$
\cite{HanHua2023}. Perrin proved analogous first eigenvalue estimates for finite graphs with boundary contained in Cayley graphs of polynomial growth \cite{Perrin2021}. 
These results do not apply to regular trees, since their volume growth is exponential.

For finite trees with leaf boundary, He and Hua obtained an order-sharp estimate in terms of the boundary cardinality and the maximum degree, together with a sharp diameter bound \cite{HHUpper2022}. They later introduced Steklov flows and established monotonicity and rigidity for the first nonzero eigenvalue \cite{HHFlows2022}. Yu and Yu extended the monotonicity theorem to higher eigenvalues on weighted graphs and obtained further estimates for trees \cite{YY2024}.

Lin and Zhao obtained bounds for planar and block graphs, and identified extremal candidates for trees with prescribed leaf number and maximum degree \cite{LZ2025b}; they also determined the extremal values for trees with prescribed leaf and vertex counts \cite{LZ2025a}. More recently, Ai-Ji-Lian-Yang obtained the optimal bound in terms of the maximum degree and the boundary cardinality, and explicitly diagonalized the Steklov operator on level-regular trees \cite{AJLY2026}. The spectrum of a regular ball is a special case of their formula.

None of these works, however, addresses the extremal problem considered here. Our focus is the far more rigid class of trees induced by connected sets $\Omega\subset T_n$: in
$\overline\Omega=\Omega\cup\delta\Omega$, every vertex of $\Omega$ has degree $n$, and the leaves are exactly $\delta\Omega$.

The rearrangement method for the  Dirichlet ground state on regular trees \cite{Pruss1998} is not directly applicable to the present Steklov problem. We instead use the inverse Dirichlet-to-Neumann form on mean-zero boundary fluxes; on a tree this form is an exact sum of squared cut fluxes. A boundary centroid, together with a lower bound for the branchwise flux energy and a finite-dimensional quadratic inequality, yields the sharp volume bound. The equality conditions force all branches at the centroid to be complete $q$-ary branches of equal depth.

Our main result is the following.
\begin{theorem}\label{thm:main}
Let $n \ge 3$ and $\Omega$ be a finite connected subset of $T_n$. Then
\[
 \sigma_1(\Omega) \leq \frac{n}{(n-1)|\Omega| + 1}.
\]
Equality holds if and only if $\Omega$ is a ball.
\end{theorem}

The proof chooses a boundary centroid $o$ and tests the inverse form on fluxes constant on the incident branches. As a corollary of Theorem \ref{thm:main}, we have the following Weinstock inequality.
\begin{corollary}\label{cor:weinstock}
Let $n\ge3$, $r\ge0$, and let $\Omega$ be a finite connected subset of $T_n$. If
\[
|B_r|\le|\Omega|<|B_{r+1}|,
\]
then
\[
\sigma_1(\Omega) \leq \frac{n-2}{(n-1)^{r+1}-1}=\sigma_1(B_r).
\]
Equality holds if and only if $\Omega=B_r$.
\end{corollary}

The paper is organized as follows. Section~\ref{sec:preliminaries} fixes the energy convention and records the inverse boundary form used throughout. Section~\ref{sec:balls} recalls the Steklov spectrum of a regular ball as a specialization of the formula for level-regular trees in \cite{AJLY2026}. Section~\ref{sec:branch-estimates} contains the boundary centroid lemma, the branchwise flux-energy lower bound, and the centroid quadratic inequality. The proof of Theorem~\ref{thm:main} is given in Section~\ref{sec:main-proof}, where the comparison between consecutive balls is also derived.

\section{Finite subtrees of the regular tree and the Steklov operator}
\label{sec:preliminaries}

Let $T_n=(V,E)$ be the infinite $n$-regular tree, $n\ge3$. For a finite connected set $\Omega\subset V$, define
\begin{equation}\label{eq:boundary-def}
\delta\Omega
 :=\{y\in V\setminus\Omega:y\sim x\text{ for some }x\in\Omega\},
 \qquad
 \overline\Omega:=\Omega\cup\delta\Omega.   
\end{equation}
The subgraph induced by $\overline\Omega$ is a finite tree whose leaves are precisely the vertices of $\delta\Omega$. Fixing $o\in V$, write
\[
 B_r(o):=\{x\in V:d(o,x)\le r\},
 \qquad
 S_r(o):=\{x\in V:d(o,x)=r\}.
\]
When the centre is immaterial, we write $B_r$ and $S_r$. Set $q=n-1$. Then
\[
\delta B_r=S_{r+1}, \qquad |S_{r+1}|=nq^r.
\]

The following identity relates boundary cardinality and vertex count.

\begin{lemma}\label{lem:boundary-volume}
For every finite connected set $\Omega\subset T_n$,
\begin{equation}\label{eq:boundary-volume}
 |\delta\Omega|=(n-2)|\Omega|+2.
\end{equation}
Consequently,
\[
 |\Omega|=|B_r|
 \quad\Longleftrightarrow\quad
 |\delta\Omega|=|S_{r+1}|.
\]
\end{lemma}

\begin{proof}
The tree induced by $\Omega$ has $|\Omega|-1$ edges. Counting the edges incident to vertices of $\Omega$ gives
\[
 n|\Omega| = 2\left(|\Omega|-1\right)+|\delta\Omega|,
\]
which is \eqref{eq:boundary-volume}. The last assertion follows by applying the identity to $\Omega$ and $B_r$.
\end{proof}

For functions $u,v:\Omega\to\R$ and $f,g:\delta\Omega\to\R$, define the inner products
\[
 \ip{u}{v}_{\Omega}:=\sum_{x\in\Omega}u(x)v(x),
 \qquad
 \ip{f}{g}_{\delta\Omega}:=\sum_{z\in\delta\Omega}f(z)g(z).
\]

We use the standard positive graph Laplacian $L=D-A$. For a function $u:\overline\Omega\to\R$, define
\[
 \Delta u(x):=\sum_{y\sim x}\left(u(x)-u(y)\right),
 \qquad x\in\Omega.
\]
For $z\in\delta\Omega$, define the discrete outward normal derivative
\[
 \partial_\nu u(z):=u(z)-u(z^-),
\]
where $z^-$ is the unique neighbour of $z$ in $\Omega$. Here $\nu$ denotes the outward normal direction at the boundary vertex $z$.

For functions $u,v:\overline\Omega\to\R$, the Dirichlet energy is
\[
 \mathcal{E}_\Omega(u,v)
 :=\sum_{\{x,y\}\in E(\overline\Omega)}
   \left(u(x)-u(y)\right)\left(v(x)-v(y)\right),
\]
where the sum is over unoriented edges, and each edge has unit conductance. Discrete summation by parts gives the Green formula
\begin{equation}\label{eq:green}
 \mathcal{E}_\Omega(u,v)
 =\ip{\Delta u}{v}_{\Omega}
  +\ip{\partial_\nu u}{v}_{\delta\Omega}.
\end{equation}

For $f:\delta\Omega\to\R$, we denote by $u_f$ the unique harmonic extension satisfying
\[
\Delta u_f=0\quad\text{on }\Omega,
\qquad
u_f=f\quad\text{on }\delta\Omega.
\]
The Dirichlet-to-Neumann operator is then defined by
\[
\Lambda_\Omega f:=\partial_\nu u_f.
\]

The following lemma collects the elementary properties needed below.

\begin{lemma}\label{lem:dtn-properties}
For every boundary function $f:\delta\Omega\to\R$, the harmonic extension $u_f$ exists and is unique. Moreover, $\Lambda_\Omega$ is self-adjoint and nonnegative,
\begin{equation}\label{eq:dtn-energy}
 \ip{\Lambda_\Omega f}{g}_{\delta\Omega}
 =\mathcal{E}_\Omega(u_f,u_g),
 \qquad
 \ip{\Lambda_\Omega f}{f}_{\delta\Omega}
 =\mathcal{E}_\Omega(u_f,u_f),
\end{equation}
and
\[
\mathrm{Ker}\,\Lambda_\Omega=\operatorname{span}\{\one\}.
\]
In addition, $\Lambda_\Omega$ leaves $\mathcal{H}_0(\delta\Omega)$ invariant and is positive definite on that subspace, where
\[
 \mathcal{H}_0(\delta\Omega)
 :=\left\{g:\delta\Omega\to\R: \sum_{z\in\delta\Omega} g(z)=0\right\}.
\]
\end{lemma}

\begin{proof}
The existence and uniqueness of the harmonic extension $u_f$ is classical; see \cite[Theorem 1.38]{Grigor2018}. We proceed to derive the remaining properties.

Apply \eqref{eq:green} to $u_f$ and $u_g$. Since $u_f$ is harmonic in $\Omega$ and $u_g=g$ on the boundary, we obtain
\[
 \mathcal{E}_\Omega(u_f,u_g)
 =\sum_{z\in\delta\Omega}(\partial_\nu u_f)(z)g(z)
 =\ip{\Lambda_\Omega f}{g}_{\delta\Omega}.
\]
The symmetry of the energy gives self-adjointness, and taking $g=f$ gives \eqref{eq:dtn-energy} and nonnegativity. If $\Lambda_\Omega f=0$, then \eqref{eq:dtn-energy} implies that $u_f$ is constant on $\overline\Omega$. Conversely, every constant function belongs to the kernel. This proves the asserted description of the kernel.

Finally, taking $v\equiv1$ in \eqref{eq:green} yields
\[
\sum_{z\in\delta\Omega}(\Lambda_\Omega f)(z)=0,
\]
that is,
\[
\ip{\Lambda_\Omega f}{\one}_{\delta\Omega}=0.
\]
Hence the range of $\Lambda_\Omega$ is contained in $\mathcal{H}_0(\delta\Omega)$. Self-adjointness then shows that $\mathcal{H}_0(\delta\Omega)=\one^\perp$ is invariant. The kernel description implies that the restriction to this subspace is positive definite.
\end{proof}

Write the eigenvalues of $\Lambda_\Omega$ as
\[
 0=\sigma_0(\Omega)<\sigma_1(\Omega)\le\cdots
 \le\sigma_{|\delta\Omega|-1}(\Omega).
\]
By the Rayleigh principle and \eqref{eq:dtn-energy}, we have
\begin{equation}\label{eq:rayleigh}
 \sigma_1(\Omega)
 =\min_{\substack{f\neq0\\ \ip{f}{\one}_{\delta\Omega}=0}}
 \frac{\ip{\Lambda_\Omega f}{f}_{\delta\Omega}}
      {\ip{f}{f}_{\delta\Omega}}
 =\min_{\substack{f\neq0\\ \ip{f}{\one}_{\delta\Omega}=0}}
 \frac{\mathcal{E}_\Omega(u_f,u_f)}
      {\ip{f}{f}_{\delta\Omega}}.
\end{equation}
 
For the extremal problem it is more convenient to invert \eqref{eq:rayleigh}. A related inverse boundary form appears in the recent work of Ai, Lin, and Shi \cite[Propositions~3.1 and~3.2]{ALS2026} on diameter-constrained Steklov problems for general trees. Their setting is more general. Nevertheless, we state and prove the result here for completeness. The exact cut decomposition and its equality case are needed in our rigidity argument.

\begin{proposition}\cite{ALS2026}\label{prop:inverse-form}
Fix $o\in \overline\Omega$ and orient every edge away from $o$. For an oriented edge $e$, let $U_e$ be the set of boundary vertices in the component of $\overline\Omega-\{e\}$ not containing $o$. Define
\begin{equation}\label{eq:cut-form}
 Q_\Omega(g)
 :=\sum_{e\in E(\overline\Omega)}
   \left(\sum_{\omega\in U_e} g(\omega)\right)^2,
 \qquad g\in\mathcal{H}_0(\delta\Omega).
\end{equation}
Then
\begin{equation}\label{eq:inverse-rayleigh}
 \frac1{\sigma_1(\Omega)}
 =\max_{0\neq g\in\mathcal{H}_0(\delta\Omega)}
  \frac{Q_\Omega(g)}{\ip{g}{g}_{\delta\Omega}}.
\end{equation}
The form $Q_\Omega$ is independent of the choice of $o$.
\end{proposition}

\begin{proof}
By Lemma \ref{lem:dtn-properties}, the restriction
\[
 L:=\Lambda_\Omega\big|_{\mathcal{H}_0(\delta\Omega)}
\]
is a positive-definite self-adjoint operator and is therefore invertible. Given $g\in\mathcal{H}_0(\delta\Omega)$, set $f=L^{-1}g$ and let $u=u_f$. Equation \eqref{eq:dtn-energy} yields
\begin{equation}\label{eq:inverse-energy}
 \ip{g}{L^{-1}g}_{\delta\Omega}
 =\ip{\Lambda_\Omega f}{f}_{\delta\Omega}
 =\mathcal{E}_\Omega(u,u).
\end{equation}

We now determine every edge difference from the boundary flux $g$. Write an oriented edge as $e=(e^-,e^+)$, where $e^-$ is the endpoint closer to $o$, and let $D_e$ be the component of
$\overline\Omega-e$ containing $e^+$. Thus $U_e=D_e\cap\delta\Omega$. Summing the discrete divergence over $D_e$ gives
\[
 \sum_{x\in D_e\cap\Omega}\Delta u(x) +\sum_{\omega\in U_e}\partial_\nu u(\omega) =u(e^+)-u(e^-).
\]
Indeed, the contributions of every edge contained in $D_e$ cancel in pairs, and the only uncancelled contribution comes from the cut edge $e$. Since $u$ is harmonic in $\Omega$ and
$\partial_\nu u=g$ on $\delta\Omega$, it follows that
\begin{equation}\label{eq:edge-flux}
 u(e^+)-u(e^-)=\sum_{\omega\in U_e} g(\omega).
\end{equation}
Consequently, summing over unoriented edges,
\[
 \mathcal{E}_\Omega(u,u)
 =\sum_{e\in E(\overline\Omega)}\left(u(e^+)-u(e^-)\right)^2
 =Q_\Omega(g).
\]
Together with \eqref{eq:inverse-energy}, this proves the operator identity
\begin{equation}\label{eq:inverse-form-identity}
 Q_\Omega(g)=\ip{g}{L^{-1}g}_{\delta\Omega},
 \qquad g\in\mathcal{H}_0(\delta\Omega).
\end{equation}

For clarity, we finish the variational step without appealing to an inverse Rayleigh principle. Choose an orthonormal basis $\{\varphi_j\}_{j=1}^{m}$ of $\mathcal{H}_0(\delta\Omega)$ consisting of Steklov eigenfunctions, where $m=|\delta\Omega|-1$ and $L\varphi_j=\sigma_j(\Omega)\varphi_j$. If $g=\sum_{j=1}^{m}c_j\varphi_j$, then
\[
 \frac{Q_\Omega(g)}{\ip{g}{g}_{\delta\Omega}}
  =\frac{\displaystyle\sum_{j=1}^{m}\frac{c_j^2}{\sigma_j(\Omega)}}{\displaystyle\sum_{j=1}^{m}c_j^2} \le\frac1{\sigma_1(\Omega)}.
\]
Equality is attained by taking $g$ in the $\sigma_1(\Omega)$-eigenspace. This proves \eqref{eq:inverse-rayleigh}.

It remains only to check that the cut form does not depend on the root. An unoriented edge $e$ partitions the boundary into two sets, say $A_e$ and $\delta\Omega\setminus A_e$. Changing the root can only replace one side by the other. Since
$g\in\mathcal{H}_0(\delta\Omega)$,
\[
 \sum_{z\in A_e}g(z)
 =-\sum_{z\in\delta\Omega\setminus A_e}g(z),
\]
and hence the squared contribution of $e$ is unchanged. This proves the last assertion.
\end{proof}

\begin{remark}\label{rem:inverse-equivalence}
Formula \eqref{eq:inverse-rayleigh} is the Rayleigh principle for $L^{-1}$, not the termwise reciprocal of \eqref{eq:rayleigh}. The two principles are equivalent because $f\mapsto g=Lf$ is a bijection of $\mathcal{H}_0(\delta\Omega)$. Under this change of variables,
\[
 \ip{g}{L^{-1}g}_{\delta\Omega} =\ip{Lf}{f}_{\delta\Omega},
\]
while the denominator becomes $\ip{Lf}{Lf}_{\delta\Omega}$. The cut identity \eqref{eq:inverse-form-identity} is the additional tree structure that makes the inverse formulation useful here.
\end{remark}

\section{The Steklov spectrum of a regular ball}
\label{sec:balls}

The spectrum of $B_r$ is a special case of \cite[Theorem~1.6]{AJLY2026}, applied to the level-regular tree of depth $r+1$ whose root has $n$ children and whose remaining interior vertices have $n-1$ children. For completeness, we state the result explicitly and give the argument in our notation. The resulting Haar-type orthogonal decomposition of the boundary will be used in the proof of the main theorem.

\begin{theorem}[{\cite{AJLY2026}}]\label{thm:ball-spectrum}
Let $n\ge3$, $q=n-1$, and $r\ge0$. The distinct nonzero Steklov eigenvalues of $B_r$ are
\begin{equation}\label{eq:ball-eigenvalues}
 \mu_\ell=\frac{q-1}{q^\ell-1},
 \qquad 1\le\ell\le r+1.
\end{equation}
Their multiplicities are
\begin{equation}\label{eq:ball-multiplicities}
 \operatorname{mult}(\mu_\ell)=
 \begin{cases}
 n-1,&\ell=r+1,\\[2mm]
 n(q-1)q^{r-\ell},&1\le\ell\le r.
 \end{cases}
\end{equation}
The eigenvalue $0$ is simple. In particular,
\[
 \sigma_1(B_r)=\mu_{r+1}
 =\frac{n-2}{(n-1)^{r+1}-1}.
\]
\end{theorem}

Fix a centre $o$ and identify $\delta  B_r$ with $S_{r+1}(o)$. For a vertex $v\in B_r$, let $\operatorname{ch}(v)$ be the set of neighbours of $v$ lying one level farther from $o$. Thus
\begin{equation}\label{eq:children-count}
 |\operatorname{ch}(v)|=
 \begin{cases}
 n,&v=o,\\
 n-1,&v\neq o.
 \end{cases}
\end{equation}
For $w\neq o$, define its descendant block by
\[
 D_w:=\{z\in S_{r+1}(o):w\in[o,z]\},
\]
where $[o,z]$ is the unique geodesic from $o$ to $z$. We also set $D_o=S_{r+1}(o)$.

For $v\in B_r$, let $\mathcal W_v$ be the space of functions $f:\delta B_r\to\mathbb{R}$ such that
\begin{enumerate}
\item $f$ vanishes outside $D_v$;
\item for every $w\in\operatorname{ch}(v)$, the restriction of $f$ to $D_w$ is a constant $c_w\in\mathbb{R}$;
\item $\sum\limits_{w\in\operatorname{ch}(v)}c_w=0$.
\end{enumerate}
The blocks $D_w$, $w\in\operatorname{ch}(v)$, have equal cardinality, so every function in $\mathcal{W}_v$ has zero sum on $D_v$. This decomposition is a discrete analogue of the spherical harmonic decomposition of the boundary of a Euclidean ball.

\begin{lemma}\label{lem:haar}
The spaces $\mathcal{W}_v$, $v\in B_r$, are mutually orthogonal, and
\begin{equation}\label{eq:haar-decomposition}
 L^2(\delta B_r)
 =\operatorname{span}\{\one\}
 \oplus\bigoplus_{v\in B_r}\mathcal{W}_v.
\end{equation}
Moreover,
\[
 \dim\mathcal{W}_v=
 \begin{cases}
 n-1,&v=o,\\
 q-1,&v\neq o.
 \end{cases}
\]
\end{lemma}

\begin{proof}
If the descendant blocks of two vertices are disjoint, the corresponding spaces are orthogonal. If $w$ is a proper descendant of $v$, then a function in $\mathcal{W}_v$ is constant on $D_w$, whereas every function in $\mathcal{W}_w$ has zero sum on $D_w$. This proves mutual orthogonality; the same zero-sum property gives orthogonality to the constants.

It remains to count dimensions. Since $|S_j(o)|=nq^{j-1}$ for $j\ge1$,
\begin{align*}
 1+\sum_{v\in B_r}\dim\mathcal{W}_v
 &=1+(n-1)+(q-1)\sum_{j=1}^r nq^{j-1}\\
 &=nq^r
 =|\delta B_r|.
\end{align*}
The orthogonal sum therefore exhausts $L^2(\delta B_r)$.
\end{proof}

\begin{lemma}\label{lem:haar-eigenvalue}
Let $v\in S_j(o)$, $0\le j\le r$, and set $\ell=r+1-j$. Then
\begin{equation}\label{eq:haar-action}
 \Lambda_{B_r}\big|_{\mathcal{W}_v}
 =\frac{q-1}{q^\ell-1}\,\operatorname{Id}_{\mathcal{W}_v}.
\end{equation}
\end{lemma}

\begin{proof}
Take $f\in\mathcal W_v$, with value $c_w$ on $D_w$ for $w\in\operatorname{ch}(v)$. Define
\[
 h_t:=\frac{1-q^{-t}}{1-q^{-\ell}},
 \qquad 0\le t\le\ell.
\]
Thus $h_0=0$, $h_\ell=1$, and
\begin{equation}\label{eq:radial-recurrence}
 (q+1)h_t-h_{t-1}-qh_{t+1}=0,
 \qquad 1\le t\le\ell-1.
\end{equation}
Define $u$ to be zero outside $D_v$ and at $v$. On the descendants of $w\in\operatorname{ch}(v)$, put
\[
 u(x):=c_w h_{d(v,x)}.
\]
Equation \eqref{eq:radial-recurrence} gives harmonicity away from $v$. Harmonicity at $v$ follows from $\sum_w c_w=0$. Hence $u$ is the harmonic extension of $f$.

For $z\in D_w$,
\begin{equation*}
 \partial_\nu u(z)
  =c_w\left(h_\ell-h_{\ell-1}\right) 
  =\frac{q-1}{q^\ell-1}\,c_w.
\end{equation*}
The normal derivative vanishes outside $D_v$, which proves \eqref{eq:haar-action}.
\end{proof}

\begin{proof}[Proof of Theorem~\ref{thm:ball-spectrum}]
By Lemmas~\ref{lem:haar} and~\ref{lem:haar-eigenvalue}, the space $\mathcal{W}_v$ with $v\in S_j(o)$ has eigenvalue $\mu_{r+1-j}$. The root contributes multiplicity $n-1$ to $\mu_{r+1}$. For $1\le\ell\le r$, the vertices in $S_{r+1-\ell}(o)$ contribute
\[
 |S_{r+1-\ell}(o)|(q-1)
 =n(q-1)q^{r-\ell}
\]
to $\mu_\ell$. The constant functions give the simple eigenvalue zero.
\end{proof}

\section{Boundary centroids and local flux-energy bounds}
\label{sec:branch-estimates}

The first lemma selects a centre according to the counting measure on the boundary, rather than according to the metric radius.

\begin{lemma}\label{lem:leaf-centroid}
Let $\Omega\subset T_n$ be finite and connected with $|\delta\Omega|\ge 3$. There exists a vertex $o\in\Omega$ such that every component $C$ of $\overline{\Omega}-\{o\}$ satisfies $|C \cap \delta \Omega| \leq |\delta\Omega|/2$.
\end{lemma}

\begin{proof}
The induced subgraph on $\overline{\Omega}$ is a finite tree whose leaves are precisely the boundary vertices $\delta\Omega$. For $x\in\overline{\Omega}$, set
\[
 F(x):=\sum_{\omega\in\delta\Omega} d(x,\omega),
\]
where $d$ is the graph distance in $T_n$. Choose $o$ minimizing $F$ over $\overline{\Omega}$.

We first show that $o\in \Omega$. Otherwise, $o\in\delta\Omega$, and let $o'$ be the unique neighbour of $o$ in $\overline{\Omega}$. For every $\omega\in\delta\Omega\setminus\{o\}$, the path from $o$ to $\omega$ passes through $o'$, so
\[
d(o',\omega)=d(o,\omega)-1,
\]
while for $\omega=o$, $d(o',o)=1=d(o,o)+1$. Thus,
\[
\begin{aligned}
F(o')-F(o)
&=\sum_{\omega\in\delta\Omega}\left(d(o',\omega)-d(o,\omega)\right)\\
&=(1-0)+\sum_{\omega\in\delta\Omega\setminus\{o\}}(-1)\\
&=2-|\delta\Omega|<0,
\end{aligned}
\]
where the last inequality follows from $|\delta\Omega|\ge3$. This contradicts the minimality of $o$. Hence, $o\in\Omega$.

Now suppose there is a component $C$ of $\overline{\Omega}-\{o\}$ such that 
\[|C\cap\delta\Omega| > |\delta\Omega|/2.\]
Let $o'\in C$ be the unique neighbour of $o$. For $\omega\in\delta\Omega$, the distance to $o'$ is $d(o,\omega)-1$ if $\omega\in C$, and $d(o,\omega)+1$ if $\omega\notin C$.
Hence
\[
\begin{aligned}
F(o')-F(o)
&=\sum_{\omega\in\delta\Omega}\left(d(o',\omega)-d(o,\omega)\right)\\
&=\sum_{\omega\in C\cap\delta\Omega}(-1) +\sum_{\omega\in\delta\Omega\setminus C}(+1)\\
&=-|C\cap\delta\Omega| +\left(|\delta\Omega|-|C\cap\delta\Omega|\right)\\
&=|\delta\Omega|-2|C\cap\delta\Omega|<0,
\end{aligned}
\]
contradicting the minimality of $o$.

Therefore every component of $\overline{\Omega}-\{o\}$ contains at most $|\delta\Omega|/2$ boundary vertices.
\end{proof}

Let $\Omega\subset T_n$ be finite and connected, and let $o\in\Omega$. For each neighbour $v$ of $o$, the corresponding branch $\mathcal{T}=\mathcal{T}(o,v)$ is the subtree of $\overline\Omega$ obtained by deleting the edge $\{o,v\}$, taking the component containing $v$, and adjoining $o$. We root $\mathcal{T}$ at $o$. Then $o$ is a boundary vertex of $\mathcal{T}$, and every interior vertex of $\mathcal{T}$ other than $o$ has exactly $q=n-1$ children.

Orient the edges of $\mathcal{T}$ away from $o$. For each $e\in E(\mathcal{T})$, let $\mathcal{T}_e$ be the component of $\mathcal{T}\setminus e$ not containing $o$, and set
\[
k_e:=|\delta\Omega\cap\mathcal{T}_e|.
\]
Define
\begin{equation}\label{eq:branch-moment-def}
W(\mathcal{T}):=\sum_{e\in E(\mathcal{T})} k_e^2.
\end{equation}

\begin{lemma}\label{lem:branch-moment}
Let $\mathcal{T}=\mathcal{T}(o,v)$ be a branch of $\overline\Omega$ as above, and $k:=|\delta\Omega\cap \mathcal{T}|$. Then
\[
W(\mathcal{T})\ge \frac{k(qk-1)}{q-1}.
\]
Equality holds if and only if, at every interior vertex of $\mathcal{T}$ other than $o$, the $q$ descendant branches contain the same number of boundary vertices. This condition forces $k=q^s$ for some $s\ge0$ and identifies $\mathcal{T}$ as the complete $q$-ary branch of depth $s+1$.
\end{lemma}

\begin{proof}
We argue by induction on the number of interior vertices of $\mathcal{T}$ other than its root.

If $v\in \delta\Omega$, then $k=1$ and $\mathcal{T}$ consists of the single edge $\{o,v\}$. Hence
\[
W(\mathcal{T})=1=\frac{1(q\cdot 1-1)}{q-1}=\frac{k(qk-1)}{q-1},
\]
and the equality condition is vacuous.

Now suppose that $v\in\Omega$. Let $v_1,\dots,v_q$ be its children, and let $\mathcal{T}_j$ be the branch rooted at $v$ with initial edge $\{v,v_j\}$. Set $k_j=|\delta\Omega\cap \mathcal{T}_j|$. Denote by $W_j=W(\mathcal{T}_j)$ its flux energy. Then
\[
k=\sum_{j=1}^q k_j,
\qquad
W(\mathcal{T})=k^2+\sum_{j=1}^q W_j,
\]
where the term $k^2$ is the contribution of the initial edge $\{o,v\}$, which separates all $k$ boundary vertices from the root $o$.

Each $\mathcal{T}_j$ has fewer interior vertices than $\mathcal{T}$, so the induction hypothesis gives
\[
W_j\ge \frac{k_j(qk_j-1)}{q-1}.
\]
Consequently,
\begin{align*}
W(\mathcal{T})-\frac{k(qk-1)}{q-1}
&\ge
k^2+\sum_{j=1}^q \frac{k_j(qk_j-1)}{q-1}
-\frac{k(qk-1)}{q-1}\\
&=\frac{q\sum_{j=1}^q k_j^2-k^2}{q-1} \ge 0,
\end{align*}
where the last inequality follows from Cauchy-Schwarz.

Equality holds preccisely when both the induction inequalities and the Cauchy-Schwarz step are equalities. This means that at every interior vertex of $\mathcal{T}$ other than $o$, the $q$ descendant branches have the same number of boundary vertices. We now identify the trees satisfying this recursive balancing. Let $z\in \mathcal{T} \cap \delta\Omega$ have minimal distance from $o$, say $s+1$. Starting from the unique child of $o$, balancing forces each descendant branch along the geodesic to $z$ to contain $k/q^j$ boundary vertices at
depth $j+1$. At $z$ this number is $1$, so $k=q^s$. Conversely, if $k=q^s$, the same balancing condition forces every vertex within distance $s$ of $o$ (except $o$) to have $q$ children, and all vertices at distance $s+1$ to be boundary vertices. Hence $\mathcal{T}$ is the complete $q$-ary branch of depth $s+1$. The converse is immediate.
\end{proof}

We shall also use the strict form of this inequality in the rigidity argument.

\begin{lemma}\label{lem:centroid-quadratic}
Let $p\ge3$, $m>0$, and let $k_1,\ldots,k_p>0$ satisfy
\begin{equation}\label{eq:centroid-hypotheses}
 \sum_{i=1}^p k_i=pm,
 \qquad
 \max_i k_i\le\frac{pm}{2}.
\end{equation}
Then
\begin{equation}\label{eq:centroid-quadratic}
 \max_{\substack{\mathbf{a}=(a_1,\dots,a_p)\in\mathbb{R}^p\setminus\{\mathbf{0}\}\\
                  \sum_{i=1}^p k_i a_i=0}}
 \frac{\sum_{i=1}^p k_i^2a_i^2}{\sum_{i=1}^p k_i a_i^2} \ge m.
\end{equation}
Equality holds if and only if $k_1=\cdots=k_p=m$.
\end{lemma}

\begin{proof}
If no $k_i$ exceeds $m$, the first condition in \eqref{eq:centroid-hypotheses} gives $k_i=m$ for every $i$, and the quotient in \eqref{eq:centroid-quadratic} is identically $m$ on the
constraint space.

Suppose that at least two of the $k_i$ are at least $m$. Unless all entries equal $m$, choose $i\neq j$ with $k_i>m$ and $k_j\ge m$. The vector
\[
 a_i=k_j,
 \qquad
 a_j=-k_i,
 \qquad
 a_\ell=0\quad(\ell\neq i,j)
\]
satisfies the constraint and has quotient
\[
\frac{2k_ik_j}{k_i+k_j}>m,
\quad\text{since } k_i>m \text{ and } k_j\ge m.
\]

It remains to consider the case in which $k_1>m$ is the unique entry larger than $m$. Put $K=\sum_{i=2}^p k_i=pm-k_1$ and take
\[
 a_1=1,
 \qquad
 a_i=-\frac{k_1}{K}\quad(2\le i\le p).
\]
By Cauchy's inequality, $\sum_{i=2}^p k_i^2\ge K^2/(p-1)$, and the corresponding quotient satisfies
\begin{align*}
 \frac{k_1\left(K^2+\sum_{i=2}^p k_i^2\right)}{Kpm}
 &\ge\frac{p k_1K}{(p-1)pm}.
\end{align*}
Write $t=k_1/(pm)$. The centroid bound gives $1/p<t\le1/2$. Since $t(1-t)$ is strictly increasing on $[1/p,1/2]$,
\[
 p^2t(1-t)>p-1.
\]
Therefore, the quotient is strictly larger than $m$.
\end{proof}

\section{Proof of the sharp inequality}
\label{sec:main-proof}

In this section, we prove Theorem~\ref{thm:main}. We begin with the following sharp estimate.

\begin{theorem}\label{thm:ball-volume}
Let $n\ge3$, $r\ge0$, and let $\Omega$ be a finite connected subset of
$T_n$. If
\[
 |\delta\Omega|=|S_{r+1}|=n(n-1)^r,
\]
then
\begin{equation}\label{eq:main-inequality}
 \sigma_1(\Omega)
 \le
 \frac{n-2}{(n-1)^{r+1}-1}
 =\sigma_1(B_r).
\end{equation}
Equality holds if and only if $\Omega$ is a ball of radius $r$.
\end{theorem}
\begin{proof} 
Set $q=n-1$ and
\[
|\delta\Omega|=nq^r.
\]
Apply Lemma \ref{lem:leaf-centroid} and let $o$ be the resulting centroid. Since $|\delta\Omega|\ge3$, the vertex $o$ lies in $\Omega$ and has degree $n$. The components of $\overline{\Omega}-\{o\}$, together with their attaching edges at $o$, form branches
\[
\mathcal{T}_1,\ldots,\mathcal{T}_n.
\]
Let $k_i=|\mathcal{T}_i\cap \delta\Omega|$. Then
\begin{equation}\label{eq:central-leaf-counts}
\sum_{i=1}^n k_i=nq^r,
\qquad
k_i\le\frac{nq^r}{2}.
\end{equation}
For each $i$, set $W_i:=W(\mathcal{T}_i)$, where $W(\cdot)$ is defined in \eqref{eq:branch-moment-def}. Given $\mathbf{a}=(a_1,\dots,a_n)\in \mathbb{R}^n$ with
$\sum\limits_i k_i a_i=0$, define a boundary flux by
\[
g(z):=a_i,
\qquad
z\in\delta\Omega\cap \mathcal{T}_i.
\]
Then
\begin{equation}\label{eq:branch-flux-form}
\ip{g}{g}_{\delta\Omega}=\sum_{i=1}^n k_i a_i^2,
\qquad
Q_\Omega(g)=\sum_{i=1}^n W_i a_i^2.
\end{equation}
By Proposition \ref{prop:inverse-form} and Lemma \ref{lem:branch-moment},
\begin{align}
\frac1{\sigma_1(\Omega)}
&\ge
\max_{\substack{\mathbf{a}\neq0\\ \sum_{i=1}^n k_i a_i=0}}
\frac{\sum_{i=1}^n W_i a_i^2}{\sum_{i=1}^n k_i a_i^2}
\notag\\
&\ge
\frac{q}{q-1}
\max_{\substack{\mathbf{a}\neq0\\ \sum_{i=1}^n k_i a_i=0}}
\frac{\sum_{i=1}^n k_i^2a_i^2}{\sum_{i=1}^n k_i a_i^2} -\frac1{q-1}. \label{eq:sharp-chain}
\end{align}
Apply Lemma~\ref{lem:centroid-quadratic} to \eqref{eq:central-leaf-counts}, with $p=n$ and $m=q^r$. This yields
\[
\frac1{\sigma_1(\Omega)}\ge
\frac{q^{r+1}-1}{q-1}.
\]
Taking reciprocals and using Theorem~\ref{thm:ball-spectrum}, we
obtain
\[
\sigma_1(\Omega)\le \frac{q-1}{q^{r+1}-1} = \sigma_1(B_r).
\]

It remains to determine the equality case. If equality holds in \eqref{eq:main-inequality}, then every inequality in the chain \eqref{eq:sharp-chain} must be an equality. Applying Lemma \ref{lem:centroid-quadratic} gives
\begin{equation}\label{eq:equal-central-branches}
k_1=\cdots=k_n=q^r.
\end{equation}
Substituting $k_i=q^r$ into Lemma \ref{lem:branch-moment} gives, for each $i$,
\begin{equation}\label{eq:branch-normalized}
\frac{W_i}{k_i}\ge\frac{q^{r+1}-1}{q-1}.
\end{equation}
We claim that equality holds for every $i$. If not, choose $i$ with
strict inequality, and choose $j\neq i$. Define a coefficient vector by
\[
a_i=1,\qquad a_j=-1,\qquad a_\ell=0\quad(\ell\neq i,j),
\]
and let $g$ be the associated boundary flux as in \eqref{eq:branch-flux-form}. Then $\sum_{\ell=1}^n k_\ell a_\ell=0$. Since $k_i=k_j=q^r$,  by \eqref{eq:branch-flux-form}, we have
\[
\frac{Q_\Omega(g)}{\|g\|^2_{\delta\Omega}}
= \frac{W_i+W_j}{k_i+k_j}
= \frac{W_i+W_j}{2q^r}
> \frac{q^{r+1}-1}{q-1}=
\frac1{\sigma_1(B_r)}. 
\]
This contradicts equality in \eqref{eq:sharp-chain}. Thus equality in \eqref{eq:branch-normalized} holds for every $i$, so every branch $\mathcal{T}_i$ attains equality in Lemma \ref{lem:branch-moment}.

Each central branch has $q^r$ boundary vertices. By the equality characterization in Lemma \ref{lem:branch-moment}, it must be the complete $q$-ary branch of depth $r+1$. The $n$ branches together therefore form $\overline{B_r(o)}$, so $\Omega=B_r(o)$. The converse follows from Theorem \ref{thm:ball-spectrum}.
\end{proof}

\begin{remark}
The centre used in the proof is determined by the distribution of the boundary vertices, not prescribed in advance. The proof involves no surgery or rearrangement of branches. In the equality case, the argument itself identifies this centroid with the centre of the regular ball.
\end{remark}

We now complete the proof of the main theorem.

\begin{proof}[Proof of Theorem \ref{thm:main}]
Let $o$ be the boundary centroid of $\Omega$ provided by
Lemma~\ref{lem:leaf-centroid}. 
The components of $\overline\Omega-\{o\}$, together with their attaching edges at $o$, form branches
\[
\mathcal{T}_1,\ldots,\mathcal{T}_n.
\]
Let $k_i=|\mathcal{T}_i\cap \delta\Omega|$. Then
\[
\sum_{i=1}^n k_i=|\delta\Omega|.
\] 
By Lemma~\ref{lem:leaf-centroid},
\[
k_i\le\frac{|\delta\Omega|}{2}
\quad\text{for every }i.
\]
Set 
\[
m:=\frac1n\sum_{i=1}^n k_i 
\]
and
\[
M:=\max_{\substack{\mathbf{a}\neq0\\ \sum_i k_i a_i=0}}
\frac{\sum_i k_i^2a_i^2}{\sum_i k_i a_i^2}.
\]
By Lemma~\ref{lem:centroid-quadratic}, we have $M\ge m=\frac{|\delta\Omega|}{n}$.

For any admissible coefficient vector $\mathbf{a}=(a_1,\dots,a_n)$, define a boundary flux $g$ by
\[
g(z):=a_i,\qquad z\in\delta\Omega\cap \mathcal{T}_i.
\]
Then
\[
\ip{g}{g}_{\delta\Omega}=\sum_{i=1}^n k_i a_i^2,
\]
and, by Lemma~\ref{lem:branch-moment},
\[
Q_\Omega(g)=\sum_{i=1}^n W(\mathcal{T}_i)a_i^2
\ge \frac{q}{q-1}\sum_{i=1}^nk_i^2a_i^2 -\frac1{q-1}\sum_{i=1}^n k_i a_i^2.
\]
Applying Proposition~\ref{prop:inverse-form}, we obtain 
\[
\frac1{\sigma_1(\Omega)}
\ge\frac{Q_\Omega(g)}{\ip{g}{g}_{\delta\Omega}}
\ge \frac{q}{q-1} \frac{\sum_i k_i^2a_i^2}{\sum_i k_i a_i^2} - \frac1{q-1}.
\]
Taking the supremum over all admissible $\mathbf{a}$ and using $M\ge \frac{|\delta\Omega|}{n}$, we obtain
\[
\frac1{\sigma_1(\Omega)}
\ge \frac{q}{q-1}M-\frac1{q-1}
> \frac{q|\delta\Omega|}{n(q-1)}-\frac1{q-1}
= \frac{(n-1)|\Omega|+1}{n},
\]
where the last equality follows from Lemma \ref{lem:boundary-volume}. Hence
\[
\sigma_1(\Omega)\leq \frac{n}{(n-1)|\Omega|+1}.
\]

We now determine the equality case. Suppose the equality holds. Then $M = \frac{|\delta \Omega|}{n}$, and for every $i$, 
\begin{equation*} 
W(\mathcal{T}_i) = \frac{k_i(qk_i-1)}{q-1}. 
\end{equation*} 
Applying Lemma \ref{lem:centroid-quadratic}, we have
\begin{equation}\label{eq:equal-general-branches} 
k_1=\cdots=k_n=\frac{|\delta \Omega|}{n}=:k. 
\end{equation} 
Moreover, Lemma \ref{lem:branch-moment} implies that
\[ 
k=q^s 
\] 
for some $s\in\mathbb N_0$, and that each $\mathcal{T}_i$ is the complete $q$-ary branch of depth $s+1$. Since all $n$ central branches have this form, their union is $\overline{B_s(o)}$. Consequently, 
\[ 
\Omega=B_s(o). 
\] 
Conversely, if $\Omega=B_s(o)$, then Theorem~\ref{thm:ball-spectrum} gives 
\[ 
\sigma_1(B_s) =\frac{q-1}{q^{s+1}-1}. 
\] 
Since 
\[ 
|B_s| =1+n\frac{q^s-1}{q-1}, 
\] 
a direct computation gives 
\[ 
\frac{n}{q|B_s|+1} =\frac{q-1}{q^{s+1}-1} =\sigma_1(B_s). 
\] 
Thus every ball attains equality.
\end{proof}

At the end of this section, we prove Corollary \ref{cor:weinstock}.

\begin{proof}[Proof of Corollary~\ref{cor:weinstock}]
Let $\Omega\subset T_n$ be finite and connected, and suppose that
\[
|B_r|\le|\Omega|<|B_{r+1}|.
\]

If $|\Omega|=|B_r|$, by Lemma \ref{lem:boundary-volume}, 
\[
|\delta\Omega|=(n-2)|\Omega|+2=(n-2)|B_r|+2=|\delta B_r|=n(n-1)^r.
\]
Theorem~\ref{thm:ball-volume} then gives
\[
\sigma_1(\Omega)\le\sigma_1(B_r),
\]
with equality if and only if $\Omega=B_r$. 
If $|B_r|<|\Omega|<|B_{r+1}|$, then Theorem \ref{thm:main} gives that $\sigma_1(\Omega) < \sigma_1(B_r)$.

Combining both cases, we obtain the asserted inequality, with equality if
and only if $\Omega=B_r$.
\end{proof}

\section*{Acknowledgments}
The authors thank Professor Bobo Hua for his guidance and constant support. L. Wang is partially supported by NSFC (Nos.~12101125 and 12371052) and the Fujian Alliance of Mathematics (No.~2024SXLMMS01). T. Wang is partially supported by the National Key R\&D Program of China (No.~2025YFA1017500).

\section*{Declarations}

\noindent\textbf{Conflict of Interest}\ \ The authors declare that they have no conflict of interest.

\vspace{1em}
\noindent\textbf{Data Availability}\ \ Data sharing is not applicable to this article as no datasets were generated or analyzed during the current study.

\bibliographystyle{plain}
\bibliography{WeinstockInequality}

\end{document}